\documentclass[12pt, a4paper]{amsart}
\usepackage{amsfonts, amssymb, amsmath}
\usepackage{amsthm}
\usepackage[margin=2.5cm]{geometry}
\usepackage{enumerate}
\usepackage{scrextend}
\usepackage[all]{xy}

\newcommand{\N}{{\mathbb N}}

\newcommand{\Q}{{\mathbb Q}}
\newcommand{\R}{{\mathbb R}}
\newcommand{\Z}{{\mathbb Z}}

\newcommand{\A}{{\mathbb A}}
\newcommand{\Cc}{{\mathbb C}}
\newcommand{\Oo}{\mathcal{O}}

\newcommand{\se}[2]{\left\lbrace #1 \mbox{ }\vline\mbox{ } #2 \right\rbrace}

\newcommand{\mul}{\mbox{mult }}
\newcommand{\mull}[1]{\mbox{mult}_{#1}}
\newcommand{\tl}[1]{\tilde{#1}}
\newcommand{\st}{^{\ast}}

\newcommand{\sep}[2]{\left\lbrace\begin{matrix} #1 & \mbox{ if } k<r, \\ #2 & \mbox{ if } k=r. \end{matrix}\right.}
\newcommand{\ol}[1]{\overline{#1}}
\newcommand{\rd}[1]{\lfloor #1\rfloor}
\newcommand{\ru}[1]{\lceil #1\rceil}
\newcommand{\w}{\mbox{wt}}
\newcommand{\tcc}[2]{\begin{tabular}{c}#1\\ #2\end{tabular}}
\newcommand{\ang}[1]{\langle #1\rangle}
\newcommand{\du}{^{\vee}}
\newcommand{\spc}{\mbox{Spec }}
\newcommand{\gq}[1]{_{\geq #1}}

\newtheorem{thm}{Theorem}[section]
\newtheorem{pro}[thm]{Proposition}
\newtheorem{cor}[thm]{Corollary}
\newtheorem{lem}[thm]{Lemma}
\theoremstyle{definition}
\newtheorem{rk}[thm]{Remark}
\newtheorem{eg}[thm]{Example}
\newtheorem{cov}[thm]{Convention}
\newtheorem*{defn}{Definition}

\begin{document}
\title{On the Nash problem for terminal threefolds of type $cA/r$}
\author{Hsin-Ku Chen}
\date{}
\address{Department of Mathematics, National Taiwan University, No. 1, Sec. 4, Roosevelt Rd., Taipei 10617, Taiwan} 
\email{d02221002@ntu.edu.tw}
\begin{abstract}
We study Nash valuations and essential valuations of terminal threefolds of type $cA/r$. If $r=1$ or the given threefold is $\Q$-factorial,
then all the Nash valuations and essential valuations can be completely described. We construct non-Gorenstein or non-$\Q$-factorial
counter examples for the Nash problem.
\end{abstract}
\maketitle
\section{Introduction}
The space of arcs is the scheme parametrize the morphisms form a formal disk to a given variety. People are interesting those arcs passing
through singularities. Ideally, all the information of the singularity are encoded in the spaces of arcs passing through singularities.
One only need to know that how to read those information.\par
Nash \cite{nash} suggests one approach. He interpret irreducible components of those arcs passing through singularities in the arc space
as a valuation near the singular locus (so called Nash valuations), and he notice that those valuations are divisorial valuations which appears
on every resolution of singularities. This correspondence is called the Nash map. The Nash problem asks whether the Nash map is an one-to-one
correspondence or not. In dimension two, it is known that the Nash problem has positive answer \cite{fp}. It is also known that if the Nash map
of a toric variety \cite{ik} or a Schubert varieties in Grassmannians \cite{dn} is bijective.
But in general the Nash map is not surjective. There are several counter examples \cite{dF}, \cite{ik}. In this paper we will also construct
two counter examples for the Nash problem.\par
In \cite{jk}, Johnson and Koll\'{a}r described the Arc space of $cA$-type isolated singularities in general dimension.
In the case of three-dimensional isolated $cA_1$ singularities, they can describe the essential valuations in detail. 
In this paper we only focus on three-dimensional singularities. Instead, we study terminal singularities of type $cA/r$, namely the
cyclic-quotient of $cA$ type singularities. The $cA/r$ singularity is one of the most common three-dimensional terminal singularities.
When running three-dimensional minimal model program, this kind of singularities occur naturally. It thus becomes an important issue
to understand $cA/r$ singularities when studying three-dimensional birational geometry.
In this paper we describe the Nash valuations of $cA/r$ singularities and we can give a explicit description for essential valuations if the
singularity is $\Q$-factorial.
\begin{thm}\label{nthm}
	Let \[X=(xy-f(z,u)=0)\subset\Cc^4_{(x,y,z,u)}/\frac{1}{r}(a,-a,1,0)\] be a $cA/r$ singularity with $r\geq 1$.
	Then there are a one-to-one correspondence between Nash valuations of $X$ and exceptional divisors over the singular point of $X$ which is
	of discrepancy less than or equal to one. More precisely, Let $w_k$ be the weight such that $w_k(z,u)=(\frac{k}{r},1)$ and let
			$m_k=\w_{w_k}f(z,u)$. Then Nash valuations of $X$ is the following valuations
	\[\se{\sigma^k_i}{\sigma^k_i(x,y,z,u)=\frac{1}{r}(\ol{ka}+ir,(m_k-i)r-\ol{ka},k,r)}\]
		for $1\leq k\leq r$ and \[ \sep{0\leq i\leq m_k-1}{1\leq i\leq m_r-1}\]
\end{thm}
For essential valuations, we can estimate its discrepancy.
\begin{pro}\label{epro}
	Let $X$ be a three-dimensional terminal singularity of type $cA/r$ with $r\geq 1$. Then every essential divisor has discrepancy
	less than or equal to two.
\end{pro}
Assume further that our singularity is $\Q$-factorial, then we can explicitly describe all the essential valuations.
\begin{thm}\label{gthm}
	Let \[X=(xy+f(z,u)=0)\subset\Cc^4\] be a $\Q$-factorial isolated $cA$ singularity.
	Let $m=\mul f(z,u)$. Then there exists a non-Nash essential valuation over $X$ if and only if
	after a suitable choice of local coordinates $(z,u)$ we have $z^iu^j\in f(z,u)$ only when $2i+j\geq 2m$.\par
	In this case there is a unique non-Nash essential valuation $\tau$ such that $\tau(x,y,z,u)=(m,m,2,1)$.
\end{thm}
\begin{thm}\label{ethm}
	Let \[X=(xy-f(z,u)=0)\subset\Cc^4_{(x,y,z,u)}/\frac{1}{r}(a,-a,1,0)\] be a $\Q$-factorial $cA/r$ singularity with $r>1$. 
	Let $w_k$ be the weight such that $w_k(z,u)=(\frac{k}{r},1)$ and let
	$m_k=\w_{w_k}f(z,u)$. Non-Nash essential valuations are those valuations: (in the following notation we assume that $a<r$)
	\[\se{\tau^k_i}{\tau^k_i(x,y,z,u)=\frac{1}{r}((k-r)a+ir,(m_k-i)r-(k-r)a,k,r)},\]
	where $r+1\leq k\leq 2r$ and $m_{k-r}-\rd{\frac{(k-r)a}{r}}\leq i\leq m_k-m_{k-r}-\ru{\frac{(k-r)a}{r}}$.\par
	In particular, the Nash map of $X$ is surjective if and only if $m_k<2m_{k-r}+(1-\delta_{k,2r})$
	for all $r+1\leq k\leq 2r$, where $\delta_{k,2r}=1$ if $k=2r$ and $0$ otherwise.
\end{thm}
Assume that the singularity is not $\Q$-factorial, then it is not easy to study essential valuations. We can only deal with Gorenstein cases:
\begin{thm}\label{noethm}
	Let \[X=(xy+f(z,u)=0)\subset\Cc^4\] be a three-dimensional isolated $cA$ singularity and assume that $X$ is not $\Q$-factorial.
	Then the Nash map of $X$ is surjective.
\end{thm}
We can construct an example which is not $\Q$-factorial and not Gorenstein, and has a non-Nash essential valuation. However
we do not have a general theory to describe the essential valuations of non-$\Q$-factorial and non-Gorenstein terminal threefolds.\par
Our result generalize Johnson-Koll\'{a}r's work in three-dimensional case and the basic idea of the proof is similar. 
The reason we only focus on dimension three cases is that one can construct an explicit resolution of a terminal threefold
using weighted blow-ups. Thus all the candidates of Nash valuations and essential valuations can be well-described. One can test whether a
valuation is Nash or not using Reguera's curve selection lemma, and test whether a valuation is essential or not using de Fernex's method.\par
In fact, usually we do not need to study every exceptional divisor on the resolution.
It is enough to study exceptional divisors on a intermediate variety which has only Gorenstein singularities.
In Section \ref{sGor} we will compute those divisorial valuations.
Section \ref{sNash} contains the main technical ingredients. We discuss the deformation of arcs on three-dimensional terminal
$cA/r$ singularities. It help us to identify Nash valuations. In Section \ref{sEss} we discuss essential valuations and we will prove
all the above theorems in Section \ref{sThm}. Counter examples of the Nash problem will also be given in the last section.\par
I want to thanks Tommaso de Fernex for discussing this question with me. I thank Jungkai Alfred Chen for his helpful comments.
Part of work was done while the author was visiting the University of Utah. The author would like to thank the University of Utah for its hospitality.
\section{Preliminary}
\subsection{Arc spaces}\label{sarc}
Let $X$ be a scheme of finite type of a field $k$. The space of arc (or the arc space) of $X$, which we will denoted by $Arc(X)$,
is a scheme satisfied the following property: for any field extension $K/k$, the $K$-valued of $Arc(X)$ is a formal arc
\[\alpha: \spc K[[t]]\rightarrow X.\] For the construction and basic properties of the arc spaces, we refer to \cite{dF2}.
We have the natural map $\pi_X:Arc(X)\rightarrow X$ which is defined by $\pi_X(\alpha)=\alpha(0)$.
If there is a morphism $f:Y\rightarrow X$, then we have a induced morphism $\pi_f:Arc(Y)\rightarrow Arc(X)$ defined by composition with $f$.\par
For a given arc $\alpha\in Arc(X)$ there is an induced morphism $\alpha\st:\Oo_X\rightarrow K[[t]]$.
Assume that \[X\cong(f_1(x_1,...,x_n)=...=f_d(x_1,...,x_n)=0)\subset\Cc^n_{(x_1,...,x_n)}.\] is an complex variety.
Every arc $\alpha\in Arc(X)$ can be express as $\alpha(t)=(x_1(t),...,x_n(t))$, where $x_i(t)=\alpha\st(x_i)\in\Cc[[t]]$, such that
$f_j(x_1(t),...,x_n(t))\equiv0$ in $\Cc[[t]]$, for all $j=1$, ..., $d$.\par
Let $\alpha\in Arc(X)$ be a arc. Then $\alpha$ induces a valuation
\[ \xymatrix{v_{\alpha}:\Oo_{X,p}\ar[r]^-{\alpha\st} & \Cc[[t]] \ar[r]^{val_t} & \Z_{\geq0}}.\]
Assume that $S\subset Arc(X)$ is a connected subset in the arc space. One can define $v_S=\min_{\alpha\in S}\{v_\alpha\}$. Since the
valuation function on the arc space is upper semi-continuous, $v_S$ is well-defined and equals to the valuation of a general element in $S$.
\begin{defn}
	Assume that $\pi_X^{-1}(X_{sing})=\bigcup_{i\in I} Z_i$, where $Z_i$ is irreducible.
	We call $\{v_{Z_i}\}_{i\in I}$ \emph{Nash valuations} of $X$.
\end{defn}
Assume that $Y\rightarrow X$ is a resolution of singularities and $\{E_j\}_{j\in J}$ are exceptional divisors of $Y\rightarrow X$.
It is known (cf. \cite[Section 3]{dF2}) that
\[\pi_f|_{\pi_Y^{-1}(\bigcup_{j\in J}E_j)}:\pi_Y^{-1}(\bigcup_{j\in J}E_j)\rightarrow \pi_X^{-1}(X_{sing})\]
is dominate. Hence for any irreducible component $Z_i$ of $\pi_X^{-1}(X_{sing})$ there exists an unique $E_j$
(uniqueness follows form the fact that $v_{E_j}\neq v_{E_{j'}}$ if $j\neq j'$) such that $\pi_f(\pi_Y^{-1}(E_j))$ dominate $Z_i$ which implies
$v_{Z_i}=v_{E_j}$. Thus Nash valuations can be viewed as a divisorial valuation which appears on every resolution of singularities of $X$.
\begin{defn}
	Let $E$ be an exceptional divisor over $X_{sing}$. $E$ is called an essential divisor if $center_YE$ is an irreducible component of
	$f^{-1}(X_{sing})$ for every resolution of singularities $f:Y\rightarrow X$. The valuation $v_E$ is called an essential valuation.
\end{defn}
The above argument yields a natural map from the set of Nash valuations to the set of essential valuations. This map is called the
\emph{Nash map}. It is obvious that the Nash map is injective (because $\pi_Y^{-1}(E_j)$ is irreducible since $\pi_Y:Arc(Y)\rightarrow Y$ is
an infinite-dimensional affine fiber bundle when $Y$ is smooth). The Nash problem asks whether the Nash map is bijective or not.
As we introduced in the first section, in some situation the Nash problem is known to be have positive answer, but in general the Nash map
is not surjective.\par
To test a divisorial valuation is a Nash valuation or not, one needs Reguera's curve selection lemma, written in the following form.
\begin{lem}[Curve selection lemma, \cite{dF2} Theorem 3.10]
	Notation as above. Assume that $\ol{\pi_f(\pi_Y^{-1}(E_{j'}))}\subsetneq \ol{\pi_f(\pi_Y^{-1}(E_j))}$,
	here the overline denotes the closure in $Arc(X)$. Then there exists a field extension
	$K/\Cc$ and a deformation of arcs $\Phi:\spc K[[s]]\rightarrow Arc(X)$ such that $\Phi(0)$ is the generic point of $\pi_f(\pi_Y^{-1}(E_{j'}))$
	and $\Phi(\eta)$ belongs to $\pi_f(\pi_Y^{-1}(E_j))/\pi_f(\pi_Y^{-1}(E_{j'}))$. Here $\eta$ denotes the generic point of $\spc K[[s]]$.
\end{lem}
\begin{cor}\label{csel}
	Notation as above. Assume that the ideal defines $f(E_{j'})$ is generated by $x_1$, ..., $x_n$.
	There exists a $\Cc$-valued deformation of arcs $\Psi:\spc\Cc[[s]]\rightarrow Arc(X)$ such that
	$\Psi(s)\in \pi_f(\pi_Y^{-1}(E_j))/\pi_f(\pi_Y^{-1}(E_{j'}))$ and $v_{\Psi(0)}(x_i)=v_{E_{j'}}(x_i)$ for all $i$.
\end{cor}
\begin{proof}
	By \cite[Lemma 7.4]{dFd}, one can choose a very general point $\Cc$-valued arc $\beta\in\pi_Y^{-1}(E_{j'})$, such that
	$\Phi$ can be restrict to $\pi_f(\beta)$. That is, there exists a $\Cc$-valued deformation of arcs $\Psi:\spc\Cc[[s]]\rightarrow Arc(X)$
	such that we have a commute diagram
	\[\xymatrix{ \spc K[[s]] \ar[rr]^{\Phi} \ar[rd] & & Arc(X) \\
		& \spc \Oo_{\pi_f(\pi_Y^{-1}(E_{j'})),\kappa_{\pi_f(\beta)}} \ar[ru] & \spc\Cc[[s]]	\ar[u]_{\Psi} \ar[l] }.\]
	\cite[Lemma 7.3]{dFd} says that one may assume $\beta(0)$ is a very general point on $E_{j'}$ and $v_{\beta}(E_{j'})=1$. 
	Thus we may assume that $v_{\Psi(0)}(x_i)=v_{\pi_f(\beta)}(x_i)=v_{E_{j'}}(x_i)$ for all $i$.
\end{proof}
Given a deformation of arcs $\Phi:\spc K[[s]]\rightarrow Arc(X)$, we will denote $\Phi_0(t)$ as the arc corresponds to the closed point and
$\Phi_{\eta}(t)$ as the arc corresponds to the generic point.
Note that $\Phi$ can be realized as a morphism $\spc K[[s,t]]\rightarrow X$, so called
a $K$-valued \emph{wedge} of $X$. We will use this notation later.

\subsection{Weighted blow-ups}
Let $G=\ang{\tau\mbox{ }\vline\mbox{ }\tau^r=id}$ be a cyclic group of order $r$. For any $\Z$-valued $n$-tuple $(a_1,...,a_n)$
one can define a $G$-action on $\A^n_{(x_1,...,x_n)}$ by $\tau(x_i)=\xi^{a_i}x_i$, where $\xi=e^{\frac{2\pi i}{r}}$.
We will denote the quotient space $\A^n/G$ by $\A^n_{(x_1,...,x_n)}/\frac{1}{r}(a_1,...,a_n)$.\par
Let $W\cong \A^n_{(x_1,...,x_n)}/\frac{1}{r}(a_1,...,a_n)$ be a cyclic-quotient singularity.
There is an elementary way to construct a birational morphism $Y\rightarrow W$, so called the weighted blow-up, defined as follows.\par
We write everything in the language of toric varieties. Let $N$ be the lattice $\ang{e_1,...,e_n,v}_{\Z}$,
where $e_1$, ..., $e_n$ is the standard basic of $\R^n$ and $v=\frac{1}{r}(a_1,...,a_n)$. Let $\sigma=\ang{e_1,...,e_n}_{\R\gq0}$.
We have $W\cong \spc\Cc [N\du\cap \sigma\du]$.\par
Let $w=\frac{1}{r}(b_1,...,b_n)$ be a vector such that $b_i=\lambda a_i+k_ir$ for $\lambda\in \N$ and $k_i\in\Z$. We define a weighted blow-up
of $W$ with weight $w$ to be the toric variety defined by the fan consists of those cones
\[\sigma_i=\ang{e_1,...,e_{i-1},w,e_{i+1},...,e_n}.\]
Let $U_i$ be the toric variety defined by the cone $\sigma_i$ and lattice $N$.
\begin{lem}\label{wbup}
	Let	\[ v'=\frac{1}{b_i}(-b_1,...,-b_{i-1},r,-b_{i+1},...,-b_n)\] and
	\[w'=\frac{1}{rb_i}(a_1b_i-a_ib_1,...,a_{i-1}b_i-a_ib_{i-1},ra_i,a_{i+1}b_i-a_ib_{i+1},...,a_nb_i-a_ib_n). \]
	Assume that $u=\frac{1}{r'}(a'_1,...,a'_n)$ is a vector such that $\ang{e_1,...,e_n,v',w'}_{\Z}=\ang{e_1,..,e_n,u}_{\Z}$,
	then \[U_i\cong \A^n/\frac{1}{r'}(a'_1,...,a'_n).\]\par
	In particular, if $\lambda=1$, then $U_i\cong\frac{1}{b_i}(-b_1,...,-b_{i-1},r,-b_{i+1},...,-b_n)$.
\end{lem}
\begin{proof}
	Let $T_i$ be a linear transformation such that $T_ie_j=e_j$ if $j\neq i$ and $T_iw=e_i$. One can see that
	\[T_ie_i=\frac{r}{b_i}(e_i-\sum_{j\neq i}\frac{b_j}{r}e_j)=v'\]
	and \[ T_iv=\sum_{j\neq i}\frac{a_j}{r}e_j+\frac{a_i}{r}\frac{r}{b_i}(e_i-\sum_{j\neq i}\frac{b_j}{r}e_j)
		=\frac{a_i}{b_i}e_i+\sum_{j\neq i} \frac{a_jb_i-a_ib_j}{rb_j}e_j=w'.\]
	Under this linear transformation $\sigma_i$ becomes the standard cone $\ang{e_1,...,e_n}_{\R\gq0}$. 
	Note that
	\begin{align*}
		k_iv'+\lambda w'&=\frac{k_ir+\lambda a_i}{b_i}e_i+\sum_{j\neq i}\frac{\lambda(a_jb_i-a_ib_j)-k_ib_jr}{rb_i}e_j \\
		&=e_i+\sum_{j\neq i}\frac{\lambda a_jb_i-b_ib_j}{rb_i}e_j=e_i-\sum_{j\neq i}k_je_j.
	\end{align*}
	Hence $e_i\in T_i N$ and $T_iN=\ang{e_1,...,e_n,u}_{\Z}$. This implies $U_i$ has cyclic quotient singularity
	which is defined by the vector $u$.\par
	Now assume that $\lambda=1$, then one can see that
	\[ w'=e_i-\sum_{j\neq i}k_je_j-k_iv',\] so one can take $u=v'$.
\end{proof}
\begin{cor}
	Let $x_1$, ..., $x_n$ be the local coordinates of $X$ and $y_1$, ..., $y_n$ be the local coordinates of $U_i$.
	The change of coordinates of $U_i\rightarrow X$ are given by $x_j=y_jy_i^{\frac{b_j}{r}}$ and $x_i=y_i^{\frac{b_i}{r}}$.
\end{cor}
\begin{proof}
	The change of coordinate is defined by $T_i^t$, where $T_i$ is defined as in Lemma \ref{wbup}.
\end{proof}
\begin{cor}
	Assume that \[S=(f_1(x_1,...,x_n)=...=f_k(x_1,...,x_n)=0)\subset W\] is a complete intersection
	and $S'$ is the proper transform of $S$ on $Y$. Assume that the exceptional locus $E$ of $S'\rightarrow S$ is irreducible and reduced.
	Then \[a(S,E)=\frac{b_1+...+b_n}{r}-\sum_{i=1}^k\w_wf_k(x_1,...,x_n)-1.\]
\end{cor}
\begin{proof}
	Assume first that $k=0$. Denote $\phi:Y\rightarrow W$. Then on $U_i$ we have
	\[\phi\st dx_1\wedge...\wedge dx_n=\frac{b_i}{r}y_i^{\frac{b_i}{r}-1}
		\left(\prod_{j\neq i} y_i^{\frac{b_j}{r}}\right) dy_1\wedge...\wedge dy_n,\]
	hence $K_Y=\phi\st K_W+(\frac{b_1+...+b_n}{r}-1)E$.\par
	Now the statement follows from the adjunction formula.
\end{proof}
It is known that any analytic germ of three-dimensional terminal singularity can be embedded into a four-dimensional cyclic-quotient space.
In this paper we are going to study $cA/r$ singularities, that is, a three-dimensional terminal singularity with the following specific form
\[ X\cong(xy-f(z,u)=0)\subset \A^4_{(x,y,z,u)}/\frac{1}{r}(a,-a,1,0)\]
\begin{cov}\label{cov}
	Assume that $X$ is of the above form and let $Y\rightarrow X$ be a weighted blow-up.
	The notation $U_x$, $U_y$, $U_z$ and $U_u$ will stand for $U_1$, ..., $U_4$ in Lemma \ref{wbup}.
\end{cov}

\subsection{Resolution of terminal threefolds of type $cA/r$}\label{sRes}
For a divisorial contraction, we always mean a birational map between terminal threefolds $f:X'\rightarrow X$, such that $exc(f)$ is an
irreducible divisor, and $K_{X'}$ is $f$-anti-ample. We say that a divisorial contraction $X'\rightarrow X$ is a \emph{$w$-morphism} if
it contracts a divisor $E$ to a point $P$, and $a(X,E)=\frac{1}{r_P}$. Here $r_P$ denotes the Cartier index of $K_X$ near $P$, that is,
the smallest integer such that $r_PK_X$ is a Cartier divisor near $P$.\par
In \cite{c} J. A. Chen proved that any terminal threefold has a \emph{feasible resolution}. That is, a sequence of $w$-morphisms
\[ X_k\rightarrow X_{k-1}\rightarrow...\rightarrow X_1\rightarrow X_0=X\]
such that $X_k$ is smooth. We will discuss the feasible resolution of $cA/r$ singularities.\par
Let \[X=(xy-f(z,u)=0)\subset\Cc^4_{(x,y,z,u)}/\frac{1}{r}(a,-a,1,0)\] be a $cA/r$ singularity. We may always assume that $z^{rm}\in f(z,u)$.
Define $w_k$ be the weight such that $w_k(z,u)=(\frac{k}{r},1)$
We denote $m_k=m_k(f)=\w_{w_k}(f(z,u))$. For convenience we will write $m_1=m$.
\begin{lem}\label{addm}
	Assume that $X'\rightarrow X$ be the weight blow-up with weight $\frac{1}{r}(a,rm-a,1,r)$ and let $P'\in X'$ be the origin of
	the chart $U_u$ (cf. Convention \ref{cov}). We denote the local equation of $U_u$ by
	\[U_u\cong(x'y'-f'(z',u)=0)\subset\Cc^4_{(x',y',z',u)}/\frac{1}{r}(a,-a,1,0).\] 
	Then $m_k(f)=m'_{k-1}(f')+m$ for all $k\geq 2$.
\end{lem}
\begin{proof}
	Write $f(z,u)=\sum a_{ij}z^{ri}u^j$, then $f'(z',u)=f(z'u^{\frac{1}{r}},u)/u^m=\sum a_{ij}{z'}^{ri}u^{i+j-m}$. Hence
	\[ m_k(f)=\min\se{ki+j}{a_{ij}\neq0}=\min\se{(k-1)i+i+j-m}{a_{ij}\neq0}+m=m'_{k-1}(f')+m.\]
\end{proof}
We now describe the feasible resolution of $X$.
\begin{enumerate}[(1)]
\item Cyclic-quotient singularities. Assume that $u\in f(z,u)$ then
	\[ X\cong\A^3_{(x,y,z)}/\frac{1}{r}(a,-a,1)\] is a cyclic-quotient singularity.
	The only $w$-morphism over $X$ is the weighted blow-up with weight $\frac{1}{r}(a,r-a,1)$. The resulting variety has two cyclic-quotient
	points of indices $a$ and $r-a$, and they are both less than $r$. By induction on $r$ we can say that, after finite steps of
	weighted blowing-ups the singularity can be resolved and we get a feasible resolution of $X$. In this case the feasible resolution is unique.
	In fact, it is the economic resolution of $X$.
\item Gorenstein $cA$ singularities. Assume that $r=1$. Since $X$ has isolated singularities, either $u^l$ or $zu^l\in f(z,u)$ for some
	$l\in\N$. Thus $m_{k+1}\leq m_k+1$ for $k\gg1$. Let $\bar{\delta}(f(z,u))=\min\se{k}{m_{k+1}\leq m_k+1}$ and
	\[\bar{\delta}(X)=\max\se{\bar{\delta}(f(z+\phi(u),u))}{\phi(u)\in u\Cc[[u]]}.\]\par
	Let $X'\rightarrow X$ be the weighted blow-up with weight $(m-1,1,1,1)$. $X'$ has a cyclic-quotient singularities which is of the form
	$\frac{1}{m-1}(-1,1,1)$ and possible some $cA$ singularities. We already known that the feasible resolution of cyclic-quotient
	singularities exists. Let $P'$ be a $cA$ point on $Y$. Since we assume that $z^{rm}\in f(z,u)$, $P'$ is not the origin of the chart $U_z$.
	After a suitable change of coordinate $z\rightarrow z+\lambda u$ one may assume that $P'$ is the origin of $U_u$.\par
	We use the notation in Lemma \ref{addm} and we denote $m'=m_1(f')$. Since we assume that $z^m\in f(z,u)$, we have ${z'}^m\in f'(z',u)$,
	hence $m\geq m'$. If $m=m'$ we have $z^{m'}\in f'(z',u)$ and Lemma \ref{addm} says that $\bar{\delta}(X)>\bar{\delta}(P'\in X')$.
	Thus we have either $m>m'$ or $m=m'$ and $\bar{\delta}(X)>\bar{\delta}(X')$. One can say that a feasible resolution of $X$
	exists by induction on the tuple $(m,\bar{\delta}(X))$.
\item $cA/r$ points with $r>1$. In this case $u^l\in f(z,u)$ for some $l\in\N$ since otherwise the singularity of $X$ is not isolated.
	Hence $m_{k+1}=m_k$ for $k\gg1$ and we define $\delta(X)=\min\se{k}{m_{k+1}=m_k}$. Note that unlike the Gorenstein case, when $r>1$
	one has $\delta(X)$ is independent of any possible change of coordinates.\par
	Let $X'\rightarrow X$ be the weighted blow-up with weight $\frac{1}{r}(a,rm-a,1,r)$.
	The origin $P'$ of the chart $U_u$ is a $cA/r$ point and the other singularities of $X'$ are cyclic quotient points and $cA$ points.
	We already known that a feasible resolution of $cA$ points and cyclic quotient points exists.
	Now we have $\delta(X)=\delta(P'\in X')-1$ by Lemma \ref{addm}, hence a feasible resolution of $P'$ exists by induction on $\delta(X)$.
\end{enumerate}

\begin{defn}
	Let \[X=(xy-f(z,u)=0)\subset\Cc^4_{(x,y,z,u)}/\frac{1}{r}(a,-a,1,0)\] be a $cA/r$ singularity. One can construct a birational map
	$Y\rightarrow X$ as follows:
	\begin{enumerate}[(1)]
	\item If $X$ is a cyclic-quotient singularity, let $Y$ be the feasible resolution (or the economic resolution) of $X$.
	\item If $X$ is a Gorenstein singularity, let $Y$ be the variety obtained by first weighted blow-up $X$ with weight $(m-1,1,1,1)$
		and then resolve all the cyclic-quotient singularities on the resulting variety in the way of step (1).
	\item If $X$ is a $cA/r$ singularity with $r>1$. Let $P_0$ be the singular point of $X_0=X$, $f_0(z,u)=f(z,u)$,
		$m^{(0)}=m_1(f_0)$. Let $X_{i+1}\rightarrow X_i$ be the weighted blow-up $P_i$ with weight $\frac{1}{r}(a,m^{(i)}r-a,1,r)$,
		$P_{i+1}$ be the origin of $(U_u)_{X_{i+1}}\subset X_{i+1}$, $xy-f_{i+1}(z,u)$ be the local defining equation near $P_{i+1}$
		and $m^{(i+1)}=m_1(f_{i+1})$. Since we have $\delta(P_{i+1})=\delta(P_i)-1$ by Lemma \ref{addm}, one has the following sequence
		of $w$-morphisms \[ X_{\delta(X)}\rightarrow X_{\delta(X)-1}\rightarrow ...\rightarrow X_1\rightarrow X_0=X\]
		such that $X_{\delta(X)}$ has only cyclic-quotient singularities or $cA$ singularities. We define $Y\rightarrow X_{\delta(X)}$
		to be the resolution of all the cyclic-quotient points on $X_{\delta(X)}$ in the way of step (1).
	\end{enumerate}
	Under this construction $Y$ is a Gorenstein terminal threefold, and we call it the \emph{Gorenstein resolution} of $X$.
\end{defn}
\section{Exceptional divisors on the Gorenstein resolution}\label{sGor}
Let $Y\rightarrow X$ be the Gorenstein resolution we constructed in the previous section.
We are going to compute the exceptional divisors on $Y$ over $X$.
\subsection{Cyclic quotient singularities}
Assume that $X\cong \Cc^3_{(x,y,z)}/\frac{1}{r}(a,-a,1)$ is a three-dimensional cyclic quotient terminal singularity.
The following statement is well-known to experts (cf. \cite[(5.7)]{r}. However we can not find a reference for the
explicit description, hence we write a proof here.
\begin{pro}\label{cq}
	There are $r-1$ exceptional divisors $E_1$, ..., $E_{r-1}$ on $Y$ over $X$. We have $a(X,E_i)=\frac{i}{r}$ and
	$E_i$ corresponds to the valuations $v_i(x,y,z)=\frac{1}{r}(\ol{ia},\ol{-ia},i)$ for $i=1$, ..., $r-1$, here $\ol{n}=n-\rd{\frac{n}{r}}r$.
\end{pro}
\begin{proof}
	We always assume $0<a<r$.
	We prove by induction on $r$. If $r=2$, then $X\cong\Cc^3/\frac{1}{2}(1,1,1)$. It is clear that after weighted blow-up $\frac{1}{2}(1,1,1)$
	we get a smooth threefold and the exceptional divisor corresponds to the valuation $v_1(x,y,z)=\frac{1}{2}(1,1,1)$. Now for general $r$, we
	consider $g:X_1\rightarrow X$ be the weighted blow-up $X$ with weight $\frac{1}{r}(a,r-a,1)$. We have $E_1=exc(X_1\rightarrow X)$
	corresponds to the valuation $v_1(x,y,z)=\frac{1}{r}(a,r-a,1)$. There are two singular point
	$P\cong \Cc^3_{(x_1,y_1,z_1)}/\frac{1}{a}(-r,r,1)$ and $Q\cong\Cc^3_{(x_2,y_2,z_2)}/\frac{1}{r-a}(r,-r,1)$. By induction on $r$ we have
	the exceptional divisors on $Y$ over $P$ corresponds to the valuations
	\[u_j(x_1,y_1,z_1)=\frac{1}{a}((\ol{-jr})_a,(\ol{jr})_a,j)=\frac{1}{a}(\ru{\frac{jr}{a}}a-jr,jr-\rd{\frac{jr}{a}}a,j).\]
	and the exceptional divisors on $Y$ over $Q$ corresponds to 
	\[w_k(x_2,y_2,z_2)=\frac{1}{r-a}((\ol{kr})_{r-a},(\ol{-kr})_{r-a},k)=
		\frac{1}{r-a}(kr-\rd{\frac{kr}{r-a}}(r-a),\ru{\frac{kr}{r-a}}(r-a)-kr,k).\]
	One only needs to show that
	\[ \{v_2,...,v_{r-1}\}=\{ u_1,...,u_{a-1}\}\cup\{w_1,...,w_{r-a-1}\}\]
	To see it, note that $E_1=(x_1=0)$ near $P$, so
	\begin{align*}
	u_j(x,y,z)&=(\frac{a}{r}\frac{1}{a}(\ru{\frac{jr}{a}}a-jr),\frac{r-a}{r}\frac{1}{a}(\ru{\frac{jr}{a}}a-jr)+\frac{1}{a}(jr-\rd{\frac{jr}{a}}a),
		\frac{1}{r}\frac{1}{a}(\ru{\frac{jr}{a}}a-jr)+\frac{1}{a}j) \\
	&=(\frac{a}{r}\ru{\frac{jr}{a}}-j,(1-\frac{a}{r})\ru{\frac{jr}{a}}-\rd{\frac{jr}{a}}+(\frac{a-r}{ra}+\frac{1}{a})jr,
		\frac{1}{r}\ru{\frac{jr}{a}})\\
	&=(\frac{a}{r}\ru{\frac{jr}{a}}-j,1-\frac{a}{r}\ru{\frac{jr}{a}}+j,\frac{1}{r}\ru{\frac{jr}{a}})\\
	&=\frac{1}{r}\left(\ol{(a\ru{\frac{jr}{a}})}_r,\ol{(-a\ru{\frac{jr}{a}})}_r,\ru{\frac{jr}{a}}\right).
	\end{align*}	 
	Similarly, we have
	\begin{align*}
	w_k(x,y,z)&=\frac{1}{r}\left(\ol{(-(r-a)\ru{\frac{kr}{r-a}})}_r,\ol{((r-a)\ru{\frac{kr}{r-a}})}_r,\ru{\frac{kr}{r-a}}\right)\\
	&=\frac{1}{r}\left(\ol{(a\ru{\frac{kr}{r-a}})}_r,\ol{(-a\ru{\frac{kr}{r-a}})}_r,\ru{\frac{kr}{r-a}}\right).
	\end{align*}
	Hence it is enough to show that
	\[\left\lbrace \ru{\frac{jr}{a}}\right\rbrace_{j=1}^{a-1} \cup \left\lbrace \ru{\frac{kr}{r-a}} \right\rbrace_{k=1}^{r-a-1}=\{2,...,r-1\}.\]
	Indeed, we have \[1<\ru{\frac{jr}{a}}<\frac{jr}{a}+1\leq\frac{(a-1)r}{a}+1=r-\frac{r}{a}+1<r.\] and similarly
	\[1<\ru{\frac{kr}{r-a}}<r.\]
	Since both left-hand-side and right-hand-side has $r-2$ elements, one only need to say that $\alpha_j=\ru{\frac{jr}{a}}$
	and $\beta_k=\ru{\frac{kr}{r-a}}$ are all distinct for $j=1$, ..., $a-1$, $k=1$, ..., $r-a-1$.
	First note that if $j<j'$ then
	\[\alpha_{j'}-\alpha_j=\ru{\frac{j'r}{a}}-\ru{\frac{jr}{a}}>\frac{j'r}{a}-(\frac{jr}{a}+1)=\frac{r}{a}(j'-j)-1\geq\frac{r}{a}-1>0,\]
	so $\alpha_{j'}\neq\alpha_j$ and similarly $\beta_{k'}\neq\beta_k$ if $k\neq k'$. Now assume that
	$\alpha_j=\beta_k=\lambda$. Let
	\[ u=\lambda a-jr=\ru{\frac{jr}{a}}a-jr,v=\lambda(r-a)-kr=\ru{\frac{kr}{r-a}}(r-a)-kr,\]
	then we have $0\leq u<a$, $0\leq v<r-a$. Thus $u+v<r$. On the other hand
	\[u+v=\lambda(a+(r-a))-(j+k)r=(\lambda-j-k)r\]
	is divisible by $r$, which implies $u+v=0$, hence $u=v=0$. We have $\lambda a=jr$ and so $\lambda a$ is divisible by $r$.
	This is impossible because $a$ and $r$ are coprime and $\lambda<r$.\par
	Now we prove that $v_i$ corresponds to a exceptional divisor of discrepancy $\frac{i}{r}$ over $X$.
	When $i=1$ this follows from the construction. Assume that $i>1$, then $v_i=u_j$ or $w_k$ for some $j$, $k$. Assume $v_i=u_j$, then
	$i=\ru{\frac{jr}{a}}$ by the computation above.
	By induction on the index we may assume $K_Y=h\st K_{X_1}+\frac{j}{a}E_i+\mbox{others}$, where $h$ denotes the morphism $Y\rightarrow X_1$.
	Hence \[ K_Y=(h\circ g)\st K_X+\frac{1}{r} h\st E_1+\frac{j}{a}E_i+\mbox{others}.\]
	Since $E_1=(x_1=0)$ and $u_j(x_1)=\frac{(\ol{-jr})_a}{a}$, we have
	\[ a(X,E_i)=\frac{1}{r}\frac{(\ol{-jr})_a}{a}+\frac{j}{a}=\frac{1}{ra}(a\ru{\frac{jr}{a}}-jr+jr)=\frac{i}{r}.\]
	Similar computation holds if $v_i=w_k$ for some $k$.	
\end{proof}
\subsection{General $cA/r$ singularities}\label{sGres}
Now let \[X=(xy+f(z,u)=0)\subset\Cc^4_{(x,y,z,u)}/\frac{1}{r}(a,-a,1,0)\] be a $cA/r$ singularity and we assume $a<r$.
Let $g:X_1\rightarrow X$ be the weighted blow-up of weight $\frac{1}{r}(a,rm-a,1,r)$.
Let $P_1$, $P_2$ and $P'$ be the origin of the charts $U_x$, $U_y$ and $U_u$ respectively. They are all possible non-Gorenstein singularities
of $X_1$. $P_1$ and $P_2$ are cyclic quotient points and $P'$ is a $cA/r$ point.\par
\begin{lem}\label{v1}
	For $j=1$, $2$, the exceptional divisor $E^j_i$ on $Y$ over $P_j$ corresponds to the valuation $v^j_i$ such that 
	\[v^1_i(x,z,u)=(\frac{a}{r}\ru{\frac{ir}{a}}-i,\frac{1}{r}\ru{\frac{ir}{a}},1)\]
	for $i=1$, ..., $a-1$ and
	\[v^2_i(y,z,u)=(\frac{rm-a}{r}\ru{\frac{ir}{rm-a}}-i,\frac{1}{r}\ru{\frac{ir}{rm-a}},1)\]
	for $i=1$, ..., $rm-a-1$.
	Furthermore we have $a(X, E^j_i)=v^j_i(z)$.
\end{lem}
\begin{proof}
	We will denote the local coordinate near $P_1\in X_1$ by $x_1$, $z_1$ and $u_1$ and we have
	$x=x_1^{\frac{a}{r}}$, $z=x_1^{\frac{1}{r}}z_1$ and $u=x_1u_1$.
	We know that $v^1_i(x_1,z_1,u_1)=\frac{1}{a}((\ol{-ir})_a,i,(\ol{ir})_a)$ by Proposition \ref{cq}, hence
	\begin{align*}
	v^1_i(x,z,u)&=(\frac{a}{r}\frac{1}{a}(a\ru{\frac{ir}{a}}-ir),\frac{1}{r}\frac{1}{a}(a\ru{\frac{ir}{a}}-ir)+\frac{i}{a},
		\frac{1}{a}(a\ru{\frac{ir}{a}}-ir)+\frac{1}{a}(ir-a\rd{\frac{ir}{a}})\\
		&=(\frac{a}{r}\ru{\frac{ir}{a}}-i,\frac{1}{r}\ru{\frac{ir}{a}},\ru{\frac{ir}{a}}-\rd{\frac{ir}{a}})\\
		&=(\frac{a}{r}\ru{\frac{ir}{a}}-i,\frac{1}{r}\ru{\frac{ir}{a}},1).
	\end{align*}
	Now we compute the discrepancy. We have $a(X_1,E^1_i)=\frac{i}{a}$ by Proposition \ref{cq}.
	Let $E$ be the exceptional divisor of $g:X_1\rightarrow X$, then $K_{X_1}=g\st K_X+\frac{1}{r}E$. Since $E$ is defined by $(x_1=0)$
	near $P_1$, $v^1_i(E)=\frac{(\ol{-ir})_a}{a}$ and
	\[a(X,E^1_i)=\frac{i}{a}+\frac{(\ol{-ir})_a}{ra}=\frac{1}{r}\ru{\frac{ir}{a}}=v^1_i(z)\]
	as the same computation in the last part of Proposition \ref{cq}.\par
	The calculation for $v^2_i$ is similar.
\end{proof}
\begin{lem}\label{mdiv}
	For a given $\lambda$ such that $1<\lambda<r$, there are exactly $m$ exceptional divisors on $Y$ over $P_1$ and $P_2$ such that
	the discrepancy of those divisors over $X$ is equal to $\frac{\lambda}{r}$.
	If $\lambda=1$ or $r$, then there are only $m-1$ exceptional divisors on $Y$ over $P_1$ and $P_2$ with discrepancy
	$\frac{\lambda}{r}$ over $X$. 
\end{lem}
\begin{proof}
	We consider the set
	\[S=\left\lbrace\ru{\frac{ir}{a}}\right\rbrace_{i=1}^{a-1}\cup\left\lbrace\ru{\frac{jr}{rm-a}}\right\rbrace_{j=1}^{rm-a-1}.\]
	One only need to show that $S$ contains $m$ elements which is equal to $\lambda$ for $1<\lambda<r$, $m-1$ elements equal to $1$ and
	$m-1$ elements equal to $r$.
	First note that every element in $S$ is a positive integer $\leq r$, and $|S|=rm-2$, as we expected. Assume that 
	$\ru{\frac{ir}{a}}=\lambda$, then \[\lambda-1<\frac{ir}{a}<\lambda,\] or equivalently \[\frac{a\lambda}{r}-\frac{a}{r}<i<\frac{a\lambda}{r}.\]
	Hence there is at most $\ru{\frac{a}{r}}=1$ many $i$ satisfied $\ru{\frac{ir}{a}}=\lambda$. Similar argument yields that
	there are at most $\ru{\frac{mr-a}{r}}=m$ many $j$ satisfied $\ru{\frac{jr}{rm-a}}=\lambda$ and at least
	$\rd{\frac{mr-a}{r}}=m-1$ many $j$ satisfied $\ru{\frac{jr}{rm-a}}=\lambda$.\par
	Let $\lambda_i=\ru{\frac{ir}{a}}$ for $i=1$, ..., $a-1$. Then $0<\lambda_1<\lambda_2<...<\lambda_{a-1}<r$.
	Assume that $\ru{\frac{jr}{rm-a}}=\lambda_i$, then we have
	\[\frac{a\lambda_i}{r}-\frac{a}{r}<i<\frac{a\lambda_i}{r}\] and \[\frac{(rm-a)\lambda_i}{r}-\frac{rm-a}{r}<j<\frac{(rm-a)\lambda_i}{r}.\]
	Hence
	\[ m\lambda_i-m=\frac{a\lambda_i}{r}-\frac{a}{r}+\frac{(rm-a)\lambda_i}{r}-\frac{rm-a}{r}<i+j<
		\frac{a\lambda_i}{r}+\frac{(rm-a)\lambda_i}{r}=m\lambda_i.\]
	For a fixed $i$ there are at most $m-1$ many $j$ satisfied this condition, hence there are exactly $m-1$ many $j$ satisfied
	$\ru{\frac{jr}{rm-a}}=\lambda_i$ and so there are exactly $m$ elements in $S$ equal to $\lambda_i$.\par
	One can see that $\ru{\frac{jr}{rm-a}}=1$ if and only if $j\leq m-1$, and $\ru{\frac{jr}{rm-a}}=r$ if and only if 
	\[\frac{jr}{rm-a}>r-1\Leftrightarrow j>rm-a-m+\frac{a}{r}\Leftrightarrow j\geq rm-a-m+1.\]
	This shows that there are exactly $m-1$ elements in $S$ equal to $1$ and $m-1$ elements in $S$ equal to $r$.
	Now there are \[rm-2-(a-1)m-2(m-1)=(r-a-1)m\] many elements in $S$ which do not equal to $1$, $\lambda_i$ or $r$. Note that
	$\{2,...,r-1\}-\{\lambda_i\}_{i=1}^{a-1}$ contains $r-a-1$ many elements and there are at most $m$ elements in
	$\left\lbrace\ru{\frac{jr}{rm-a}}\right\rbrace_{j=1}^{rm-a-1}$ which have the same value. This implies there are exactly $m$ elements
	in $S$ with value $\lambda$ for $\lambda\in\{2,...,r-1\}-\{\lambda_i\}_{i=1}^{a-1}$. 
\end{proof}
Recall that when $r>1$ we have defined 
\begin{pro}\label{sig}
	Given a positive integer $k$.
	\begin{enumerate}[(1)]
	\item If $r=1$ then there are exactly $m-1$ exceptional divisors on $Y$ over $X$. They correspond to the valuations
		$\sigma^k_i(x,y,z,u)=(i,m-i,1,1)$ for $i=1$, ..., $m-1$. 
	\item Assume that $r>1$ and $k\leq r$ (resp. $k=r$). There are exactly $m_k$ (resp. $m_r-1$) many exceptional divisors
		on $Y$ over $X$ which is of discrepancy $\frac{k}{r}$. They correspond to the valuations
		$\sigma^k_i(x,y,z,u)=\frac{1}{r}(\ol{ka}+ir,(m_k-i)r-\ol{ka},k,r)$ for $i=0$, ..., $m_k-1$ (resp. $i=1$, ..., $m_r-1$).
	\item Assume that $r>1$, $k=r+k_0>r$ and $k_0\leq\delta'(X)$. Then there are exactly $m_k-m_{k_0}-1$ many exceptional divisors
		on $Y$ over $X$ which is of discrepancy $\frac{k}{r}$. They correspond to the valuations
		$\tau^k_i(x,y,z,u)=\frac{1}{r}(k_0a+ir,(m_k-i)r-k_0a,k,r)$ for $i=1$, ..., $m_k-m_{k_0}-1$.
	\item If $k_0>\delta(X)$, then there is no exceptional divisor on $Y$ over $X$ which is of discrepancy $\frac{k}{r}$.
	\end{enumerate}
\end{pro}
\begin{proof}
	To prove (1) and (2) we only need the following observations:
	\begin{enumerate}[(i)]
	\item The total number of exceptional divisor on $Y$ of discrepancy $\frac{k}{r}$ is $m_k$ (resp. $m_r-1$)
		if $k<r$ (resp. $k=r$).
	\item If $E\subset Y$ is an exceptional divisor of discrepancy $\frac{k}{r}$, then $v_E(z,u)=(\frac{k}{r},1)$ and either
		$rv_E(x)\equiv ka(mod\mbox{ }r)$ or $rv_E(y)\equiv -ka(mod\mbox{ }r)$.
	\end{enumerate}
	When $r=1$ the statement follows from Lemma \ref{v1}. If $r>1$ it is easy to check that (i) and (ii) is true by using Lemma \ref{v1},
	Lemma \ref{addm} and by induction on $\delta'(X)$.\par
	From now on we will assume that $r>1$. First we prove (3). One can construct a sequence of $w$-morphisms
	\[ X_{\delta'(X)}\rightarrow...\rightarrow X_1\rightarrow X_0=X\]
	such that in each step we contract a divisor to a $cA/r$ point $P'_j\in X_j$.
	Note that we have $\delta'(P'_j\in X_j)=\delta'(P'_{j+1}\in X_{j+1})+1$.
	Assume that \[(P'_j\in X_j)\cong (x_jy_j+f_j(z_j,u_j)=0)\subset \A^4_{(x_j,y_j,z_j,u_j)}/\frac{1}{r}(a,-a,1,0).\]
	Let $w_k^{(j)}$ be the weight such that $w_k^{(j)}(z_j,u_j)=(\frac{k}{r},1)$. Define $m^{(j)}_k=\w_{w^{(j)}_k}f_j(z_j,u_j)$.
	By Lemma \ref{addm}, we have $m^{(j)}_k=m^{(j+1)}_{k-1}+m^{(j)}_1$.
	Let $n^{(j)}_k$ be the number of exceptional divisors on $Y$ over $X_j$ which is of discrepancy $\frac{k}{r}$.
	Note that Lemma \ref{v1} implies any exceptional divisor over $P_1$ and $P_2$ has discrepancy less than or equal to one.
	By Lemma \ref{v1} and by induction on $\delta(X)$ one can show that $v_E(u)=1$ for all exceptional divisor $E$ on $Y$ over $X$.
	One can compute that $a(X_j,E)=a(X_{j+1},E)+\frac{1}{r}$ for all $j$. The conclusion is that when $k>r$ and $j<\delta(X)$,
	we have $n^{(j)}_k=n^{(j+1)}_{k-1}$.\par
	Now
	\begin{align*}
	n^{(0)}_k=n^{(1)}_{k-1} =...&=n^{(k_0)}_r\\
	&=m^{(k_0)}_r-1\\&=m^{(k_0-1)}_{r+1}-m^{(k_0-1)}_1-1\\&=m^{(k_0-2)}_{r+2}-m^{(k_0-2)}_1-m^{(k_0-1)}_1-1\\
	&=m^{(k_0-2)}_{r+2}-m^{(k_0-2)}_2-1\\&=...\\&=m^{(0)}_k-m^{(0)}_{k_0}-1
	\end{align*}
	If $E$ is an exceptional divisor of discrepancy $\frac{k}{r}$, then $a(X_j,E)=\frac{k-j}{r}$ for all $j\leq k_0$.
	Hence $a(X_{k_0},E)=1$ and $v_E(x_{k_0},y_{k_0},z_{k_0},u_{k_0})=(i,m^{(k_0)}_r-i,1,1)$ and one can check that
	$v_E(x_j,y_j,z_j,u_j)=\frac{1}{r}((k_0-j)a+ir,(m^{(j)}_{k-j}-i)r-(k_0-j)a,k-j,r)$. Thus
	\[v_E(x,y,z,u)=\frac{1}{r}(k_0a+ir,(m_k-i)r-k_0a,k,1).\]
	Finally if $k_0>\delta(X)$, then $n_k=n^{(\delta(X))}_{k-\delta(X)}=0$ since $k-\delta(X)>r$.
\end{proof}

\section{Nash valuations of terminal singularities of type $cA/r$}\label{sNash}
As before we assume that \[X=(xy-f(z,u)=0)\subset\Cc^4_{(x,y,z,u)}/\frac{1}{r}(a,-a,1,0)\] is a $cA/r$ singularity.
We use the notation in Section \ref{sarc}.
\begin{lem}\label{disz}
	Let \[\Psi:\Oo_X\rightarrow\Cc[[s,t]]\] be a deformation of arcs on $X$. 
	Assume that $\mull{t}\Psi_0(x^r)$, $\mull{t}\Psi_0(y^r)$, $\mull{t}\Psi_0(z^r)$ and $\mull{t}\Psi_0(u)$ are all finite
	and $\mull{t}\Psi_0(u)=\mull{t}\Psi_{\eta}(u)$, then
	\[ v_{\Psi_0}(z)-v_{\Psi_{\eta}}(z)\in \Z.\]
	In particular, if $E$ and $E'$ are two exceptional divisors such that $\pi_f(\pi_Y^{-1}(E))\subset\pi_f(\pi_Y^{-1}(E'))$
	and $a(X,E)=v_E(z)$, $a(X,E')=v_{E'}(z)$ (for example if $E\subset Y$, where $Y$ is the Gorenstein resolution of $X$),
	then $a(X,E)-a(X,E')\in\Z$.
\end{lem}
\begin{proof}
	Note that $\mull{t}\Psi_0(x^r)$, $\mull{t}\Psi_0(y^r)$ and $\mull{t}\Psi_0(z^r)$ are finite implies $\mull{t}\Psi_0(xz^{r-a})$,
	$\mull{t}\Psi_0(yz^a)$ and $\mull{t}\Psi_0(xy)=\mull{t}\Psi_0(f(z,u))$ are all finite. For a fixed integer $n$, define
	\[ \xymatrix@R=2pt{\Psi_n:\Oo_X \ar[r]^{\Psi} & \Cc[[s,t]]\ar[r] &\Cc[[s,t]] \\ & s\ar@{|->}[r] & {s}^n}.\]
	By Newton's Lemma \cite[Lemma 7]{jk}, there exists $n$ such that
	\[\Psi_n(x^r)=\xi(s,t)t^{\alpha}\prod_i(t-\sigma_i(s))^{\lambda_i};\quad
		\Psi_n(z^r)=\eta(s,t)t^{\gamma}\prod_j(t-\tau_j(s))^{\mu_j}\]
	such that $\xi(0,0)$ and $\eta(0,0)\neq0$, $\sigma_i(0)=\tau_j(0)=0$ but $\sigma_i(s)$ and $\tau_j(s)$ are not identically zero.
	We may assume that similar factorizations exist for $\Psi_n(y^r)$ and $\Psi_n(xy)$.\par
	We show that $r$ divide $\mu_j$ for all $j$, which implies
	\[v_{\Psi_0}(z^r)-v_{\Psi_{\eta}}(z^r)=\left(\gamma+\sum_j\mu_j\right)-\gamma \in r\Z,\]
	hence $v_{\Psi_0}(z)-v_{\Psi_{\eta}}(z)\in \Z$.
	Indeed, since $\Psi_n(x^r)\Psi_n(z^r)^{r-a}=\Psi_n(xz^{r-a})^r$,
	if $r$ do not divide $\mu_j$, then there exists $i$ such that $\sigma_i(s)=\tau_j(s)$ and $r$ do not divide $\lambda_i$.
	We have $t-\sigma_i(s)$ divides \[\Psi_n(x^r)\Psi_n(y^r)=\Psi_n(xy)^r=\Psi_n(f(z,u))^r,\]
	Hence $t-\sigma_i(s)$ divides $\Psi_n(f(z,u))$.
	Since $t-\sigma_i(s)=t-\tau_j(s)$ divides $\Psi_n(z^r)$ and $z$ do not divide $f(z,u)$ because $X$ has isolated singularities,
	we have $t-\sigma_i(s)$ divide $\Psi_n(u)$. However it is impossible since $\mull{t}\Psi_0(u)=\mull{t}\Psi_{\eta}(u)$.\par
	Now the last statement follows from Corollary \ref{csel} and Proposition \ref{sig}.
\end{proof}
\begin{lem}\label{subn}
	Let $E$ be an exceptional divisor over $X$ such that 
	\begin{enumerate}
	\item $v_E(z)=a(X,E)=\frac{k}{r}>1$.
	\item $v_E(x)=\frac{i}{r}$ with $i\equiv ka\mbox{ (mod $r$)}$.
	\item $v_E(x)+v_E(y)=rm_k$.
	\item $v_E(u)=1$.
	\end{enumerate}		
	Then $v_E$ is not a Nash valuation.
\end{lem}
\begin{proof}
	Let $\mu:\Oo_X\rightarrow \Cc[[t]]$ be an arc such that $v_{\mu}=v_E$.
	We may write $\mu(x^r)=\alpha(t)^rt^i$, $\mu(y^r)=\beta(t)^rt^j$ and $\mu(z^r)=\gamma(t)^rt^k$, where $\alpha(t)$, $\beta(t)$
	and $\gamma(t)$ are units. Note that we have $i\equiv ka\mbox{ (mod $r$)}$ and $i+j=m_k$.
	We define $z_s(t)=t^{\frac{k}{r}-1}(\gamma(t)t+s)\in\Cc[[s,t^{\frac{1}{r}}]]$, $u_s(t)=\mu(u)$ and
	$F(s,t)=f(z_s(t),u_s(t))\in\Cc[[s,t]]$. By Newton's Lemma \cite[Lemma 7]{jk} there exists an integer $n$ and a factorization
	\[F(s^n,t)=\phi(s,t)t^{m_{k-r}}\prod_{l=1}^{m_k-m_{k-r}}(t-\sigma_l(s)),\quad \phi(0,0)\neq0,\quad \sigma_l(0)=0
		\mbox{ but }\sigma(s)\not\equiv0.\]
	We can choose $i'$ and $j'$ satisfying $i'\equiv ka\mbox{ (mod $r$)}$, $i'+j'=rm_{k-r}$ and $i'\leq i$, $j'\leq j$.
	Define \[x_s(t)=\alpha(t)t^{\frac{i'}{r}}\prod_{l=1}^{\frac{i-i'}{r}}(t-\sigma_l(s))\in\Cc[[s,t^{\frac{1}{r}}]]\]
	and\[ y_s(t)=\alpha(t)^{-1}t^{\frac{j'}{r}}\prod_{l=\frac{i-i'}{r}+1}^{m_k-m_{k-r}}(t-\sigma_l(s))\in\Cc[[s,t^{\frac{1}{r}}]].\]
	Now we can define a deformation of arcs $\Psi:\Oo_X\rightarrow\Cc[[s,t]]$ as follows: given $x^py^qz^vu^w\in\Oo_X$, define
	$\Psi(x^py^qz^vu^w)=x_s(t)^py_s(t)^qz_s(t)^vu_s(t)^w$. It is easy to see that $v_{\Psi_0}=v_E$. Thus $E$ is not a Nash valuation.
\end{proof}
Combining the two above lemmas one may conclude the following.
\begin{pro}\label{nashcar}
	Let $Y\rightarrow X$ be the Gorenstein resolution of $X$.	
	Assume that $E\subset Y$ is an exceptional divisor over $X$. Then $E$ corresponds to a Nash valuation of $X$
	if and only if $a(X,E)\leq 1$.
\end{pro}
\section{Essential valuations of terminal singularities of type $cA/r$}\label{sEss}
Assume that \[X=(xy-f(z,u)=0)\subset\Cc^4_{(x,y,z,u)}/\frac{1}{r}(a,-a,1,0)\] is a $cA/r$ singularity and $Y\rightarrow X$ is the Gorenstein resolution of $X$.
\subsection{General situation}
\begin{lem}\label{es1}
	Let $E$ be an exceptional divisor over $X$.
	Assume that \begin{enumerate}[(i)]
		\item $v_E(z)=\frac{k}{r}>1$ and $v_E(u)=1$.
		\item $v_E(x)<m_{k-r}$ and $rv_E(x)\equiv a\mbox{ mod }r$, or $v_E(y)<m_{k-r}$ and $rv_E(y)\equiv -a\mbox{ mod }r$
	\end{enumerate}
	Then $E$ is not an essential divisor. In particular, if $E\subset Y$ and $a(X,E)>2$, then $E$ is not essential.
\end{lem}
\begin{proof}
	We write $v_E(x,y)=\frac{1}{r}(\alpha,\beta)$ and we may assume $\alpha<m_{k-r}r$ and $\alpha\equiv a\mbox{ mod }r$.
	Let $X'\rightarrow X$ be the weighted blow-up
	with weight $w$ such that $w(x,y,z,u)=\frac{1}{r}(\alpha,m_{k-r}r-\alpha,k-r,r)$. Then the chart $U_u\in X'$
	has only isolated singularities. One can see that $Center_{X'}(E)$ is a curve, hence $E$ can not be essential.\par
	Now assume that $E\subset Y$ and $a(X,E)>2$. We will show that $m_k<2m_{k-r}$. In this case Proposition \ref{sig} implies both (i) and (ii)
	are true, hence $E$ can not be essential.\par
	To see that $m_k<2m_{k-r}$, let $z^iu^j$ be the monomial in $f(z,u)$ such that $\w_{w_{k-r}}(z^iu^j)=m_{k-r}$. We have
	$\frac{k-r}{r}i+j=m_{k-r}$ and
	\[m_k\leq \frac{k}{r}i+j=m_{k-r}+i\leq m_{k-r}+\frac{r}{k-r}m_{k-r}=\left(1+\frac{r}{k-r}\right)m_{k-r}<2m_{k-r},\]
	by noticing that \[i=\frac{r}{k-r}\left(m_{k-r}-j\right)\leq\frac{r}{k-r}m_{k-r}.\]
\end{proof}

\subsection{$Q$-factorial cases}
\begin{lem}\label{es2}
	Assume that $X$ has only $\Q$-factorial singularities. Let $E$ be an exceptional divisor over $X$ such that 
	\begin{enumerate}[(i)]
		\item $1<v_E(z)=a(X,E)\leq 2$.
		\item $v_E(u)=1$.
		\item Both $v_E(x)$ and $v_E(y)\geq m_{k-r}$.
	\end{enumerate}
	Then $E$ is an essential divisor.
\end{lem}
\begin{proof}
	Assume that $E$ is not an essential divisor, then there exists a smooth model $g:Z\rightarrow X$ such that $\Gamma=Center_Z(E)$
	is not a divisor (Note that $exc(Z\rightarrow X)$ is pure of codimension one under the assumption that $X$ is $\Q$-factorial,
	cf. \cite[Lemma 17]{jk}). Let $F\subset Z$ be an exception divisor containing $\Gamma$. We may write $K_Z=g\st K_X+\Delta$. By
	\cite[Lemma 2.29]{km}, we have \[a(X,E)=a(Z,-\Delta,E)\geq codim_Z(\Gamma)-1+a(X,F)\mbox{mult}_{\Gamma} F.\]
	By Lemma \ref{disz} we have $a(X,E)-a(X,F)\in\Z$. Since $a(X,E)\leq 2$, $a(X,F)=a(X,E)-1$.
	Note that we have already assume that $\Gamma$ is not a divisor, hence $codim_Z(\Gamma)\geq 2$. If $\mbox{mult}_{\Gamma} F\geq 2$, then
	\[codim_Z(\Gamma)-1+a(X,F)\mbox{mult}_{\Gamma} F\geq 1+2a(X,E)-2>a(X,E)\]
	since $a(X,E)>1$, which leads a contradiction. Thus $\mbox{mult}_{\Gamma} F=1$ and $codim_Z(\Gamma)=2$. This says that
	$\Gamma$ is a curve and $F$ is smooth along $\Gamma$ generically.\par
	Now we have $a(X,F)\leq 1$, hence $Center_YF$ is a divisor because $Y$ is Gorenstein. By Proposition \ref{sig}
	we have $v_F(x,y,z,u)=\frac{1}{r}(i,m_{k-r}r-i,k-r,r)$ for some positive integer $i<m_{k-r}r$. Let $X'\rightarrow X$ be the weighted blow-up
	with weight $v_F$. We are going to show that the rational map $\mu:Z\dashrightarrow X'$ is well-defined along generic point of
	$\Gamma$ and $Center_{X'}E$ is a point $Q$. Thus $\mu$ contract $\Gamma$ to a point but maps $F$ to the exceptional divisor of
	$X'\rightarrow X$. We have $\mu^{-1}(Q)$ is not pure of codimension one. However, since $X$ is $\Q$-factorial, $X'$ is $\Q$-factorial.
	This contradict to \cite[Lemma 17]{jk}.\par
	To see that $\mu$ is well-defined along the generic point of $\Gamma$, consider an affine open set $U_Z$ on $Z$ such that $F$ is defined by
	$v=0$ for some regular function $v$ on $U_Z$. Since $val_F(x,y,z,u)=\frac{1}{r}(\alpha,m_{k-r}r-\alpha,k-r,r)$,  
	One may write $x=v^{\frac{i}{r}}\alpha$, $y=v^{m_{k-r}-\frac{i}{r}}\beta$, $z=v^{\frac{k-r}{r}}\gamma$ and $u=v\delta$,
	such that $\alpha$, $\beta$, $\gamma$ and $\delta$ do not vanish along $F$. Furthermore since $v_E(u)=v_F(u)=1$, $\delta$ do not vanish
	along the generic point of $\Gamma$.\par
	On the other hand, the regular functions near the origin of $U_u\subset X'$ is generated by $x'=x/u^{\frac{i}{r}}$,
	$y'=y/u^{m_{k-r}-\frac{i}{r}}$, $z'=z/{u}^{\frac{k-r}{r}}$ and $u$.
	Thus the coordinate change of $Z\dashrightarrow X'$ is given by $x'=\alpha/\delta^{\frac{i}{r}}$, $y'=\beta/\delta^{m_{k-r}-\frac{i}{r}}$,
	$z'=\gamma/\delta^{\frac{k-r}{r}}$ and $u=v\delta$. This shows that $Z\dashrightarrow X'$ is well-defined along the generic point of $\Gamma$.
	Furthermore, since $v_E(x)>v_F(x)$, $v_E(y)>v_F(y)$ and $v_E(z)>v_F(z)$, we have $\alpha$, $\beta$ and $\gamma$ vanish along $\Gamma$.
	Thus the image of $\Gamma$ on $X'$ is the origin of the chart $U_u$.
\end{proof}
\begin{pro}\label{gEss}
	Assume that $X$ has Gorenstein $\Q$-factorial $cA$ type singularities. Let $E$ be an exceptional divisor over $X$.
	Then $E$ is a non-Nash essential divisor if and only if $a(X,E)=2$, $2m=m_2$ and $v_E(x,y,z,u)=(m,m,2,1)$
	for a suitable choice of local coordinates of $X$.
\end{pro}
\begin{proof}
	First assume that $a(X,E)=2$ and $v_E(x,y,z,u)=(m,m,2,1)$. In this case $E$ is essential by Lemma \ref{es2}. One can see that $E$ is not Nash
	by applying Lemma \ref{subn}.\par
	Now we assume that $E$ is a non-Nash essential divisor and we are going to prove that the above conditions hold.
	Let $X_1\rightarrow X$ be the weighted blow-up with weight $(m-1,1,1,1)$ and let $P'=Center_{X_1}E$. 
	$P'$ should be a singular point of $X_1$. Note that $X_1$ has one cyclic-quotient point and other possible singularities are $cA$
	singularities. The Gorenstein resolution of $X$ is obtained by resolving the cyclic-quotient point of $X_1$. Hence if $P'$ is the
	cyclic-quotient point, then $E$ must appear on the Gorenstein resolution of $X$. However in this case $E$ should correspond to a Nash
	valuation of $X$ by Proposition \ref{sig} and Proposition \ref{nashcar}. Hence $P'$ is a Gorenstein point on $X_1$.\par
	Note that $E$ is also an essential divisor of $P'$.
	After a suitable change of coordinates on $X$ one may assume that
	$P'$ is the origin of the chart $U_u$. Let $x'$, $y'$, $z'$ and $u$ be the local coordinate near $P'$ and let $x'y'+f'(z',u)$ be the
	local defining equation of $U_u$. Let $m'=\mul f'(z',u)$. As we discussed in Section \ref{sRes}, $P'\in X_1$
	has better singularity the $X$ in the sense that the tuple $(m,\bar{\delta}(X))>(m',\bar{\delta}(P'\in X_1))$.
	We will induction on this tuple, and assume that our statement hold for $E$ over $P'\in X_1$.
	More precisely, we may assume that either $E$ corresponds to a Nash valuation of $P'$, which implies $a(X_1,E)=1$, or
	$a(X_1,E)=2$ and $v_E(x',y',z'',u'')=(m',m',2,1)$ for a suitable change of coordinate $z''=\alpha(z',u)$ and $u''=\beta(z',u)$.\par
	First we assume that $a(X_1,E)=1$. Proposition \ref{sig} says that $v_E(x',y',z',u)=(i,m'-i,1,1)$ for some
	$i<m'$, where $m'=m_2-m$ by Lemma \ref{addm}. This implies that $v_E(x,y,z,u)=(m-1+i,1+m'-i,2,1)$. Note that we have $m_2\leq 2m$ and
	$m'\leq m$. Lemma \ref{es1} says that $v_E(x)=m-1+i\geq m$ and $v_E(y)=m'+1-i\geq m$. Hence $i=1$ and $m'=m$. we have
	$m_2=2m$ and $v_E(x,y,z,u)=(m,m,2,1)$. One can also compute that $a(X,E)=2$.\par
	Now we assume that $a(X_1,E)=2$ and $v_E(x',y',z'',u'')=(m',m',2,1)$. First assume that $v_E(u)=1$. In this case we may assume that
	$u=u''$ and $z''=z'+u\phi(u)$ for some $\phi(u)\in\Cc[[u]]$. Let $z_1=z+u^2\phi(u)$ we can see that $z_1=z''u$ and
	$v_E(x,y,z_1,u)=(m+m'-1,m'+1,3,1)$. Also notice that in this case we have $a(X,E)=3$. Hence $E$ is not essential by Lemma \ref{es1}.\par
	Finally assume that $v_E(u)=2$. We may assume $u=z''$ and $z'=u''$. We have $v_E(x,y,z,u)=(m'+2(m-1),m'+2,3,2)$.
	Let $k=\ru{\frac{m'}{2}}$ and define $\bar{X}\rightarrow X$ be the weighted blow-up with weight $(m-k-1,k+1,1,1)$.
	Let $(\bar{y},\bar{z},\bar{u})$ be the local coordinate of the chart $\bar{U}_y\subset \bar{X}$. We have
	\[v_E(\bar{y},\bar{z},\bar{u})=\frac{1}{k+1}(m'+2,3(k+1)-m'-2,2(k+1)-m'-2).\]
	Since $2(k+1)\geq m'+2$ one can see that $Center_{\bar{X}}(E)$ is either a cyclic-quotient point or a curve.
	If it is a curve then $E$ can not be essential. If it is a cyclic-quotient point then $E$ either corresponds to a
	Nash valuation of $X$, or is not essential. This proves our statement.
\end{proof}

\subsection{Non-$\Q$-factorial cases}\label{sness}
Let \[X=(xy-f(z,u)=0)\subset\Cc^4_{(x,y,z,u)}/\frac{1}{r}(a,-a,1,0)\] be a $cA/r$ singularity. Let $f(z,u)=f_1(z,u)...f_n(z,u)$ be a
factorization into irreducible components in $\Cc[[z^r,u]]$, here $f_i(z,u)$ are all invariant under the cyclic action.
By \cite[2.2.7]{k}, we have
\[Pic(X-O)\cong \Z/r\Z[K_{X-O}]+\frac{\Z[x=f_1(z,u)=0]+...+\Z[x=f_n(z,u)=0]}{([x=f_1(z,u)=0]+...+[x=f_n(z,u)=0])}.\]
In particular, $X$ is $\Q$-factorial if and only if $f(z,u)$ is irreducible in $\Cc[[z^r,u]]$.\par
We follow the construction in \cite[Section 2.2]{k} to construct a $\Q$-factorization of $X$.
Let $\tl{X}=(xy-(z,u)=0)\subset\A^4$ be the canonical cover of $X$ and let $G=<\sigma>$ be the cyclic group such that $X=\tl{X}/G$.
Let $\tl{X_1}$ be the blow-up of the ideal $(x,f_1(z,u))$ on $\tl{X}$. There are two affine charts on $\tl{X_1}$. They are
\[ U_s=\left(\tcc{$xs-f_1(z,u)=0$}{$y-sf_2(z,u)...f_n(z,u)=0$}\right)\cong(xs-f_1(z,u)=0)\subset\A^4_{(x,s,z,u)}\]
and \[U_t=\left(\tcc{$tf_1(z,u)-x=0$}{$ty-f_2(z,u)...f_n(z,u)=0$}\right)\cong(ty-f_2(z,u)...f_n(z,u)=0)\subset\A^4_{(t,y,z,u)}.\]
One may define $G$-action on $\tl{X_1}$ by $\sigma(s)=\sigma(y)$ and $\sigma(t)=\sigma(x)$ and let $X_1=\tl{X_1}/G$.
$\tl{X_1}$ also has $cA/r$ singularities. We denote the image of the origin of the chart $U_s$ by $P$ and the image of the origin of the chart
$U_t$ by $P'$. Then $P$ is a $\Q$-factorial point. $P'$ may not be $\Q$-factorial, but it has better singularity than $X$ in the sense that the
number of irreducible components of the defining equation decreases. Repeat this process we get a sequence of terminal threefolds with
$cA/r$ singularities \[ X'=X_{n-1}\rightarrow ...\rightarrow X_1\rightarrow X_0=X\]
such that $X'$ has $\Q$-factorial singularities. Note that $\tl{X_1}\rightarrow \tl{X}$ is isomorphic in codimension one, hence
$X_1\rightarrow X$ is isomorphic in codimension one. Inductively we have $X'\rightarrow X$ is isomorphic in codimension one. Thus $X'$ is in fact
a $\Q$-factorization of $X$.\par
Let $C_i=exc(X_i\rightarrow X_{i-1})$ and $C'_i$ by the proper transform of $C_i$ on $X'$. Recall that we define $w_k$ be the weight such that
$w_k(z,u)=(\frac{k}{r},1)$. For any $g(z,u)\in\Cc[[z,u]]$, we define $m_k(g)=\w_{w_k}g(z,u)$.
\begin{lem}
	Nash valuations of $X$ is the union of the Nash valuations of $X'$ and the valuation obtained by blowing-up $C'_i$.
\end{lem}
\begin{proof}
	Nash valuations of $X$ corresponds to exceptional divisors of discrepancy less than or equal to one. Since $X'\rightarrow X$ is isomorphic
	in codimension one, for any exceptional divisor $E$ over $X'$ we have $a(X',E)=a(X,E)$. Given $k<r$, Proposition \ref{sig} says that there
	are exactly $m_k(f)$ many exceptional divisors over $X$ which is of discrepancy $\frac{k}{r}$. Note that there are $n$ $cA/r$ points on
	$X'$ which is defined by $(xy-f_i(z,u)=0)\subset\Cc^4_{(x,y,z,u)}/\frac{1}{r}(a,-a,1,0)$ for $i=1$, ..., $n$. Thus the total number of
	exceptional divisors of discrepancy $\frac{k}{r}$ over $X'$ is $\sum_{i=1}^n m_k(f_i)=m_k(f)$. This says that the exceptional divisors
	which is of discrepancy less than one over $X$ is exactly those exceptional divisors of discrepancy less than one over $X'$.\par
	Now we count the number of discrepancy one exceptional divisors. There are $m_r(f)-1$ many exceptional divisors of discrepancy one
	over $X$ and $\sum_{i=1}^{n}(m_r(f_i)-1)=m_r(f)-n$ many exceptional divisors of discrepancy one over $X'$. Let $E_i$ be the
	exceptional divisor obtained by blowing-up $C'_i$, for $i=1$, ..., $n-1$, then we also have $a(X,E_i)=a(X',E_i)=1$. Thus the
	exceptional divisors of discrepancy one over $X$ are exactly the exceptional divisors of discrepancy one over singular points of $X'$,
	plus $\{E_i\}_{i=1}^{n-1}$. This proves the lemma.
\end{proof}
\begin{pro}\label{nQ}
	Assume that $r=1$ and $X$ is not $\Q$-factorial.
	Then the Nash map of $X$ is surjective.
\end{pro}
\begin{proof}
	Let $E$ be an essential divisor of $X$. Then $Center_{X'}E$ is either a singular point of $X'$ or $C'_i$.
	If $Center_{X'}E=C'_i$, then $E$ should be the blow-up of $C'_i$, so $E$ corresponds to a Nash valuation of $X$.
	Assume that $Center_{X'}E$ is a singular point of $X'$, then $E$ is an essential divisor of $X'$.
	If $E$ corresponds to a Nash valuation of $X'$, then $E$ corresponds to a Nash valuation of $X$ so there is nothing to do.
	Now we assume that $E$ is a non-Nash essential divisor of $X'$ and we will show that this it impossible.\par
	Let $P_i=Center_{X_i}E$ and let $j$ be the smallest integer such that $P_j$ has $\Q$-factorial singularity. We are going to say that
	$E$ is not a essential divisor of $P_{j-1}$, hence $E$ can not be an essential divisor of $X$. Thus we may assume $j=1$.\par
	As the notation above $P_1$ is defined by
	\[ (xs-f_1(z,u)=0)\subset \A^4_{(x,s,z,u)}\] such that $y=sf_2(z,u)...f_n(z,u)$.
	We have $E$ is an essential divisor of $P_1$. By Proposition \ref{gEss} we have $val_E(x,s,z,u)=(m_1(f_1),m_1(f_1),2,1)$,
	hence \[val_E(x,y,z,u)=(m_1(f_1),m_1(f_1)+\sum_{i=2}^nm_2(f_i),2,1).\]
	However, Lemma \ref{es1} says that $E$ can not be an essential divisor.
\end{proof}
\begin{rk}
	When $r>1$ there is an example such that $X$ is not $\Q$-factorial and the Nash map is not surjective, please see Example \ref{eg}.
	However we can not find a general theory to describe all the essential valuations.
\end{rk}
\subsection{Valuations over the Gorenstein resolution}
\begin{pro}\label{esscar}
	Assume that $X$ has $cA/r$ singularity with $r>1$, then every exceptional divisor over $Y$ is not an essential divisor of $X$.
\end{pro}
\begin{proof}
	Let $F$ be an exceptional divisor over $Y$ and $Q=Center_YF$. If $Q$ is a curve or $Q$ is a smooth point, then $F$ can not
	be an essential divisor. Now we may assume $Q$ is a $cA$ point.\par
	Recall that we have the sequence of divisorial contractions
	\[Y=X_k\rightarrow X_{k-1}\rightarrow ...\rightarrow X_1\rightarrow X_0=X.\]
	Let $Q_i=Center_{X_i}F$. Let $j$ be the smallest index such that $Q_j$ is a Gorenstein point on $X_j$, then $Q_{j-1}$ is a
	non-cyclic-quotient $cA/r$ point. We are going to prove that $F$ is not an essential divisor of $X_{j-1}$, and hence $F$ can not be a
	essential divisor of $X$. For simplicity we may assume $j=1$.\par
	After suitable change of coordinate, we may assume that $Q$ is the origin of the chart $U_z$.
	If $F$ is not a essential divisor of $X_1$, then $F$ can not be a essential divisor of $X$ and we have done.
	Assume now that $F$ is essential over $X_1$. We want to find a birational morphism $\bar{X}\rightarrow X$ such that $\bar{X}$
	has isolated singularities and $Center_{\bar{X}}F$ is a curve. This will imply $F$ is not an essential divisor of $X$.\par
	The construction of $\bar{X}$ is as follows. Assume that $x'$, $y'$, $z'$ and $u'$ are local coordinates near $Q$ and
	$f'(z',u')$ is the local defining equation. Let $m'=\mul f'(z',u')$. There are two possibilities.
	\begin{enumerate}[(i)]
	\item $a(X_1,F)=1$. In this case $v_F(x',y',z',u')=(c,m'-c,1,1)$ for some positive integer $c<m'$.
		We have $v_F(x,y,z,u)=\frac{1}{r}(cr+a,(m'+m)r-cr-a,1,2r)$. Note that $m\geq m'$, hence $cr+a<mr$.
		Let $\bar{X}\rightarrow X$ be the weighted blow-up with weight $\frac{1}{r}(cr+a,mr-cr-a,1,r)$. 
	\item $a(X_1,F)=2$. Note that $X_1$ has an essential divisor of discrepancy two implies $X_1$ has $\Q$-factorial singularities by
		Proposition \ref{nQ}. Thus the defining equation of $X_1$ satisfied the condition in Proposition \ref{gEss}.
		This will implies $f'(z',u')_{m'}=(pz'+qu')^{m'}$ for some $p$, $q\in\Cc$, where $f'(z',u')_{m'}$ denotes the homogeneous part of
		degree $m'$ of $f'(z,u')$.
		\begin{enumerate}[({ii-}1)]
		\item $q\neq0$. This implies $u^{m'}\in f(z,u)$, hence $m=m'$.
			After suitable change of coordinates we may assume $f'(z,u')_{m'}={u'}^{m'}$ and
			$v_F(x',y',z',u')=(m',m',1,2)$. Thus $v_F(x,y,z,u)=\frac{1}{r}(rm+a,2rm-a,1,3r)$. Let $\bar{w}$ be the weight such that
			$\bar{w}(z,u)=(\frac{1}{r},2)$. Let $\bar{m}=\w_{\bar{w}}f(z,u)$. Since $u^m\in f(z,u)$, $\bar{m}\leq 2m$.
			Since $z^{rm}\not\in f(z,u)$ (or the origin of $U_z$ do not contained in $X_1$), we have $m<\bar{m}$, hence $rm+a<r\bar{m}$.
			One can define $\bar{X}\rightarrow X$ to be the weighted blow-up with weight $\frac{1}{r}(rm+a,r(\bar{m}-m)-a,1,2r)$.
		\item $q=0$. Hence $f'(z',u')_{m'}={z'}^{m'}$ and $v_F(x',y',z',u')=(m',m',2,1)$. One can see that
			$v_F(x,y,z,u)=\frac{1}{r}(rm'+2a,r(m'+2m)-2a,2,3r)$. We need to check that\[m'\leq m_2-2.\]
			Indeed, there exists a monomial $z^{ir}u^j\in f(z,u)$ such that $i+j=m$. Assume that $i\neq0$.
			Since $i+2j-m=j>m'$, we have $m>j>m'$, hence $m_2\geq m\geq m'+2$ and we have done. Now assume $f(z,u)_{m}=u^m$.
			In this case $f'(z',u')={z'}^{m'}+{u}^{m}+$others. The condition that $X_1$ has a discrepancy two essential valuation implies
			$m\geq 2m'$. Since $m'\geq2$, $m_2=m\geq m'+2$.\par
			Note that $2(m+m')>2m\geq m_2$. We can define $\bar{X}\rightarrow X$ to be the weighted blow-up with weight
			$\frac{1}{r}(rm'+2a,r(m_2-m')-2a,2,r)$.
		\end{enumerate}
	\end{enumerate}		
	
\end{proof}
\section{Proof of the main theorems}\label{sThm}

\begin{proof}[Proof of Theorem \ref{nthm}]
	Proposition \ref{sig}, Proposition \ref{nashcar} and Proposition \ref{esscar} implies our theorem.
\end{proof}
\begin{proof}[Proof of Proposition \ref{epro}]
	If $r=1$, Proposition \ref{gEss} and Proposition \ref{nQ} implies the statement. When $r>1$,
	it follows from Lemma \ref{es1} and Proposition \ref{esscar}.
\end{proof}
\begin{proof}[Proof of Theorem \ref{gthm}]
	It is Proposition \ref{gEss}.
\end{proof}
\begin{proof}[Proof of Theorem \ref{ethm}]
	Proposition \ref{sig}, Lemma \ref{es1}, Lemma \ref{es2} and Proposition \ref{esscar} implies the theorem.
\end{proof}
\begin{proof}[Proof of Theorem \ref{noethm}]
	It is Proposition \ref{nQ}.
\end{proof}
\begin{eg}
	Let \[X=(xy+z^r+u^{2r}=0)\subset\Cc^4_{(x,y,z,u)}/\frac{1}{r}(1,-1,1,0).\] Then $X$ is $\Q$-factorial since
	$z^r+u^{2r}$ is irreducible in $\Cc[[z^r,u]]$ (cf. Section \ref{sness}). We have $m_k=k$ for all $k\leq 2r$.
	Thus Nash valuations of $X$ are
	\[ \se{\sigma^k_i}{\sigma^k_i(x,y,z,u)=\frac{1}{r}(k+ir,(k-i)r-k,k,r);1\leq k\leq r-1,0\leq i\leq k-1} \]
	\[\quad\cup \se{\sigma^r_i}{\sigma^r_i(x,y,z,u)=(i,r-i,1,1);1\leq i\leq r-1 }\]
	and non-Nash essential valuations of $X$ are
	\[\se{\tau^{r+k}_i}{\tau^{r+k}_i(x,y,z,u)=\frac{1}{r}(k+ir,(r+k-i)r-k,r+k,r);1\leq k\leq r-1,k\leq i\leq r-1}.\]
	In particular, there are \[\sum_{k=1}^{r-1}k+r-1=(r-1)\left(\frac{r}{2}+1\right)\] many Nash valuations, and
	\[(r-1)\left(\frac{r}{2}+1\right)+\sum_{k=1}^{r-1}(r-k)=(r-1)\left(\frac{r}{2}+1\right)+r(r-1)-\frac{r(r-1)}{2}=(r-1)(r+1)=r^2-1\]
	many essential valuations.
\end{eg}
\begin{eg}\label{eg}
	Let \[X=(xy+(z^6+u^{11})(z^2+u)=0)\subset\Cc^4_{(x,y,z,u)}/\frac{1}{2}(1,1,1,0).\] It is a non-$\Q$-factorial $cA/2$ singularity.
	We have $m_1=4$ and $m_3=10$. Thus there exists an exceptional divisor $E$ over $X$ such that $v_E(x,y,z,u)=\frac{1}{2}(9,11,3,2)$
	by Proposition \ref{sig}. Since $a(X,E)=\frac{3}{2}$, $E$ do not correspond to a Nash valuation.
	We are going to show that $E$ is an essential divisor. Thus the Nash map of $X$ is not surjective.\par
	Assume that $E$ is not essential. Then there exists a smooth model $W\rightarrow X$ such that $Center_WE=\Gamma$ is a curve,
	$\Gamma\subset F$ for some exceptional divisor of discrepancy $\frac{1}{2}$ and $F$ is smooth along $\Gamma$ (cf. the first
	paragraph in the proof of Lemma \ref{es2}). We have $v_F(x,y,z,u)=\frac{1}{2}(a,b,1,2)$ such that $a+b=8$.
	We may write $x=\alpha t^{\frac{a}{2}}$, $y=\beta t^{\frac{b}{2}}$, $z=\gamma t^{\frac{1}{2}}$ and $u=\delta t$, for some $\alpha$, $\beta$
	$\gamma$, $\delta$ and $t\in\Oo_W(U)$ such that $U$ is a affine open set contains $\Gamma$ and $t$ is the local defining function of $F$.
	Note that since $v_F(x)<v_E(x)$, $v_F(y)<v_E(y)$, $v_F(z)<v_E(z)$ but $v_F(u)=v_E(u)$, we have $\alpha$, $\beta$ and $\gamma$ vanish on
	$\Gamma$ but $\delta$ do not vanish near $\Gamma$.\par
	We may assume $b\geq3$. Let $X_1$ be the blowing-up the ideal $(x,z^6+u^{11})$. As the computation in Section \ref{sness} there is a chart
	$U_s\in X_1$ which is defined by
	\[ U_s=\left(\tcc{$xs-(z^6+u^{11})=0$}{$y-s(z^2+u)=0$}\right)\cong(xs-(z^6+u^{11})=0)\subset\A^4_{(x,s,z,u)}/\frac{1}{2}(1,1,1,0).\]
	One can see that $v_F(x,s,z,u)=\frac{1}{2}(a,b-2,1,2)$. Let $X'\rightarrow X_1$ be the weighted blowing-up the origin of $U_s$ with this
	weight. Consider the chart $U'_u\subset X'$. The local coordinate of $X'$ is given by 
	\[x'=x/u^{\frac{a}{2}}=\alpha/\delta^{\frac{a}{2}},\quad
		s'=s/u^{\frac{b-2}{2}}=y/(z^2+u)u^{\frac{b-2}{2}}=\beta/(\gamma^2+\delta)\delta^{\frac{b-2}{2}},\quad
		z'=z/u^{\frac{1}{2}}=\gamma/\delta^{\frac{1}{2}}\]
	and $u=\delta t$. One can see that there is a rational map from $W$ to $X'$ which maps $F$ to a divisor but maps $\Gamma$ to the origin.
	However, since $X_1$ is $\Q$-factorial, $X'$ is $\Q$-factorial. This leads a contradiction.
\end{eg}

\end{document}